\documentclass[12pt]{amsart}
\usepackage{amsmath,amssymb}      
\usepackage{ascmac}
\usepackage{mathrsfs}
\usepackage[abbrev]{amsrefs}
\usepackage{comment}

\usepackage{amsthm}
\usepackage{latexsym}
\usepackage{mathtools}
\usepackage{amsfonts}
\usepackage{bm}

\allowdisplaybreaks[1]
\theoremstyle{definition}
\newtheorem{thm}{Theorem}[section]

\newtheorem{prop}[thm]{Proposition}
\newtheorem{lem}[thm]{Lemma}
\newtheorem{dfn}[thm]{Definition}

\newtheorem*{dfn*}{Definition}
\newtheorem*{rem*}{Remark}

\newenvironment{pr}
   {{\noindent \bf Proof.   }}{\hfill \qed}

\def\<{\langle }
\renewcommand{\<}{\langle}
\renewcommand{\>}{\rangle }

\newcommand{\eqsp}[1]{{\begin{equation}\begin{split}#1\end{split}\end{
equation}}}

\usepackage{graphicx,xcolor}

\begin{document}


\begin{center}
\textbf{\Large{Higher-order derivative estimates for the parabolic Lam\'{e} system on a smooth bounded domain}}
\end{center}
\vskip10mm

\centerline{Yoshinori Furuto${}^*$ and Tsukasa Iwabuchi${}^{**}$}
\vskip5mm
\centerline{Mathematical Institute, Tohoku University}
\centerline{Sendai 980-8578 Japan}
\footnote[0]{\it{Mathematics Subject Classification}: 35K40; 35K50; 76N06}
\footnote[0]{\it{Keywords}: Lam\'{e} system, derivative estimates, smooth bounded domain}
\footnote[0]{E-mail: $^*$yoshinori.furuto.p3@dc.tohoku.ac.jp,
$^{**}$t-iwabuchi@tohoku.ac.jp}

\begin{center}
\begin{minipage}{120mm}
We consider the parabolic Lam\'{e} system on a bounded domain.
We focus on two types of inequalities for higher-order derivatives of solutions.
The first is related to an $L^p$-$L^p$ estimate locally in time in the Lebesgue space setting, 
which includes the endpoint cases $p=1$ and $p=\infty$.
The second concerns an equivalent norm of Besov spaces by means of the solution of the parabolic Lam\'{e} system.
\end{minipage}
\end{center}

\section{Introduction}

Let $d \geq 2$. Suppose that $\Omega$ is a bounded domain of $\mathbb{R}^d$.
We consider the parabolic Lam\'{e} system with the Dirichlet boundary condition.
\begin{equation} \label{eq:L}
\begin{cases}
\partial_t u - (\mu \Delta + (\mu + \lambda)\nabla \text{div}) u = 0
&\quad \text{in } \ (0, T) \times \Omega, \\
u =  0
&\quad \text{on } \ (0, T) \times \partial \Omega, \\
u(0, \cdot) = u_0(\cdot) 
&\quad \text{in } \ \Omega.
\end{cases}
\end{equation}
The unknown function is $u = u(t,x) \in \mathbb{R}^d$, 
and the constants $\mu$ and $\lambda$ denote the shear viscosity and the bulk viscosity, respectively.
Moreover, due to the ellipticity of the operator $\mu \Delta + (\mu + \lambda)\nabla \operatorname{div}$, we assume that 
$\mu > 0$ and $2\mu + \lambda > 0$ (see~\cite[Section~1.1]{paper:MiMo-2010} for more details).

The parabolic Lam\'{e} system originates from the linearization of the momentum equation of the compressible Navier--Stokes equations
\begin{equation}
\partial_t (\rho u) + \text{div}(\rho u \otimes u) - \mu \Delta u - (\mu + \lambda)\nabla \text{div} u + \nabla P = 0
\quad \text{ in } (0, T) \times \Omega
\end{equation}
around the constant state $(\rho, u) = (1,0)$ after neglecting the pressure term.
We refer to~\cite{paper:MiMo-2010} and \cite{paper:DaTo-2022}, 
for references summarizing properties of the Lam\'{e} operator on $L^p(\Omega)$ for $p \in (1,\infty)$.
In this paper, we investigate higher-order derivative estimates for solutions to~\eqref{eq:L}, 
including the end-point cases $p = 1, \infty$.

We define the Lam\'{e} operator $- \mathcal{L}$ on $L^2(\Omega)$ by  
\begin{equation*}
\begin{cases}
\mathcal{D}(\mathcal{L}) = H^1_0(\Omega) \cap 
\left\{ u \in L^2(\Omega) \mid (\mu \Delta + (\mu + \lambda)\nabla \text{div}) u \in L^2(\Omega) \right\}, 
\\
-\mathcal{L} u = -(\mu \Delta + (\mu + \lambda)\nabla \text{div}) u , \quad u \in \mathcal{D}(\mathcal{L}).
\end{cases}
\end{equation*}
Then $- \mathcal{L}$ is a non-negative self-adjoint operator and 
the semigroup on $L^2(\Omega)$ generated by $-\mathcal{L}$ is defined via the spectral theorem as
\begin{equation*}
\begin{aligned}
e^{t\mathcal{L}}
= \int_{0}^{\infty} e^{-t \lambda} \, \mathrm{d}E_{- \mathcal{L}}(\lambda).
\end{aligned}
\end{equation*}
Here, $\{E_{- \mathcal{L}}(\lambda)\}_{\lambda \in \mathbb{R}}$ denotes the resolution of the identity 
uniquely determined by the spectral theorem. 
For $\psi \in C((0, \infty)) \cap L^\infty(0, \infty)$, $\psi(-\mathcal{L})$ on $L^2(\Omega)$ is defined by
\begin{equation*}
\psi(-\mathcal{L})
= \int_{0}^{\infty} \psi(\lambda) \, \mathrm{d}E_{- \mathcal{L}}(\lambda).
\end{equation*}

As in the paper~\cite{paper:IwMaTa-2019} and \cite{paper:Iw-2018}, 
we introduce the distributions and Besov spaces associated with the Lam\'{e} operator. 
Let $\{\phi_j\}_{j\in\mathbb{Z}}$ be the dyadic Littlewood--Paley partition of unity:
\begin{equation*}
\begin{aligned}
&\phi_0(\cdot) \in C_c^\infty(\mathbb{R}), 
\quad \text{supp} \, \phi_0 \subset \left\{ \lambda \in \mathbb{R} \mid \frac{1}{2} \leq \lambda \leq 2 \right\}, 
\quad \sum_{j \in \mathbb{Z}} \phi_0(2^{-j} \lambda) = 1 \quad \text{ for } \lambda > 0, 
\\
&\phi_j(\lambda) \coloneqq \phi_0(2^{-j} \lambda) \quad \text{ for } \lambda \in \mathbb{R}, 
\end{aligned}
\end{equation*}
and define the tempered distributions as follows.

\begin{dfn}
\noindent(i) $\mathcal{X}(\mathcal{L})$ is defined by letting
\begin{equation*}
\mathcal{X}(\mathcal{L}) \coloneqq \{ f \in L^1(\Omega) \cap \mathcal{D}(\mathcal{L}) 
\mid (-\mathcal{L})^M f \in L^1(\Omega) \cap \mathcal{D}(\mathcal{L}) \quad \text{ for all } M \in \mathbb{N} \},
\end{equation*}
equipped with the family of seminorms $\{ p_{M}(\cdot) \}_{M=1}^{\infty}$ given by
\begin{equation*}
p_{M}(f) \coloneqq \| f \|_{L^1(\Omega)} 
+ \sup_{j \in \mathbb{N}} 2^{Mj} \| \phi_j(\sqrt{-\mathcal{L}}) f \|_{L^1(\Omega)}.
\end{equation*}
$\mathcal{X}'(\mathcal{L})$ denotes the topological dual of $\mathcal{X}(\mathcal{L})$.

\noindent(ii) $\mathcal{Z}(\mathcal{L})$ is defined by letting
\begin{equation*}
\mathcal{Z}(\mathcal{L}) \coloneqq \{ f \in \mathcal{X}(\mathcal{L}) 
\mid \sup_{j \leq 0} 2^{M |j|} \| \phi_j(\sqrt{-\mathcal{L}}) f \|_{L^1(\Omega)} < \infty \quad \text{ for all } M \in \mathbb{N} \}, 
\end{equation*}
equipped with the family of seminorms $\{ q_{M}(\cdot) \}_{M=1}^{\infty}$ given by
\begin{equation*}
q_{M}(f) \coloneqq \| f \|_{L^1(\Omega)} 
+ \sup_{j \in \mathbb{Z}} 2^{M |j|} \| \phi_j(\sqrt{-\mathcal{L}}) f \|_{L^1(\Omega)}.
\end{equation*}
$\mathcal{Z}'(\mathcal{L})$ denotes the topological dual of $\mathcal{Z}(\mathcal{L})$.
\end{dfn}

In order to treat operators on more general spaces, 
We define mapping on $\mathcal{X}'(\mathcal{L})$ and $\mathcal{Z}'(\mathcal{L})$.

\begin{dfn} \label{def:12-012}
Let $\psi \in C((0, \infty)) \cap L^\infty(0, \infty)$.

\noindent
(i) For mapping $\psi(\mathcal{L}) : \mathcal{X}(\mathcal{L}) \to \mathcal{X}(\mathcal{L})$, 
we define $\psi(\mathcal{L}) : \mathcal{X}'(\mathcal{L}) \to \mathcal{X}'(\mathcal{L})$ by letting
\begin{equation*}
\begin{aligned}
{}_{\mathcal{X}'}\< \psi(\mathcal{L})f, g \>_{\mathcal{X}}
\coloneqq {}_{\mathcal{X}'}\< f, \psi(\mathcal{L})g \>_{\mathcal{X}} 
\quad \text{ for all } g \in \mathcal{X}(\mathcal{L}).
\end{aligned}
\end{equation*}

\noindent
(ii) For mapping $\psi(\mathcal{L}) : \mathcal{Z}(\mathcal{L}) \to \mathcal{Z}(\mathcal{L})$, 
we define $\psi(\mathcal{L}) : \mathcal{Z}'(\mathcal{L}) \to \mathcal{Z}'(\mathcal{L})$ by letting
\begin{equation*}
\begin{aligned}
{}_{\mathcal{Z}'}\< \psi(\mathcal{L})f, g \>_{\mathcal{Z}}
\coloneqq {}_{\mathcal{Z}'}\< f, \psi(\mathcal{L})g \>_{\mathcal{Z}} 
\quad \text{ for all } g \in \mathcal{Z}(\mathcal{L}).
\end{aligned}
\end{equation*}
\end{dfn}

We consider semigroups $e^{t \mathcal{L}}$ on $\mathcal{X}'(\mathcal{L})$ and $\mathcal{Z}'(\mathcal{L})$ via this construction.
Note that Lebesgue spaces $L^p(\Omega)$ are embedded into $\mathcal{X}'(\mathcal{L})$ and $\mathcal{Z}'(\mathcal{L})$.

We first consider the $L^p$--$L^p$ estimates for higher-order derivatives.  
When the domain is the whole space $\Omega = \mathbb{R}^d$, 
the solution to (L) can be written explicitly as 
\begin{equation*}
\begin{aligned}
u(t) = K * u_0, 
\quad &K_{ij}(t, x) \coloneqq G_{\mu t} \delta_{ij} + \partial_i \partial_j \int_{\mu t}^{(2\mu + \lambda)t} G_s \, \mathrm{d}s, 
\\
&G_t(x) \coloneqq \dfrac{1}{(4 \pi t)^{d/2}} \exp{\left( -\dfrac{|x|^2}{4t} \right)}.
\end{aligned}
\end{equation*}
In fact, by Fourier transform, 
\begin{equation*}
\begin{aligned}
\partial_t \widehat{u} 
&= - (\mu |\xi|^2 I + (\mu + \lambda) (\xi \otimes \xi)) \widehat{u}
\\
&= - \left( \mu |\xi|^2 \left(I - \frac{\xi \otimes \xi}{|\xi|^2} \right) 
  + (2 \mu + \lambda) (\xi \otimes \xi) \right) \widehat{u}, 
\end{aligned}
\end{equation*}
which implies from orthogonality of the divergence-free part and the curl-free part, 
\begin{equation*}
\begin{aligned}
\widehat{u}(t, \xi)
&= e^{-\mu |\xi|^2 t} \left(I - \frac{\xi \otimes \xi}{|\xi|^2} \right) \widehat{u_0}(\xi) 
  + e^{- (2\mu + \lambda) |\xi|^2 t} \frac{\xi \otimes \xi}{|\xi|^2} \widehat{u_0}(\xi)
\\
&= e^{-\mu |\xi|^2 t} \widehat{u_0}(\xi) 
  + \frac{\xi \otimes \xi}{|\xi|^2} (e^{- (2\mu + \lambda) |\xi|^2 t} - e^{-\mu |\xi|^2 t}) \widehat{u_0}(\xi)
\\
&= e^{-\mu |\xi|^2 t} \widehat{u_0}(\xi) 
  + \frac{\xi \otimes \xi}{|\xi|^2} 
  \left( \int_{\mu t}^{(2\mu + \lambda) t} \partial_s(e^{-|\xi|^2 s}) \, \mathrm{d}s \right) \widehat{u_0}(\xi)
\\
&= e^{-\mu |\xi|^2 t} \widehat{u_0}(\xi)
  - (\xi \otimes \xi) \int_{\mu t}^{(2\mu + \lambda) t} e^{- |\xi|^2 s} \widehat{u_0}(\xi) \, \mathrm{d}s.
\end{aligned}
\end{equation*}
Therefore, a direct computation yields the following estimate:
\begin{equation*}
\|  \partial_x^\gamma u(t) \|_{L^p(\mathbb{R}^d)} \leq C t^{-|\gamma| / 2} \| u_0 \|_{L^p(\mathbb{R}^d)}
\quad \textup{ for all } t \in (0, 1), 
\end{equation*}
where $1 \leq p \leq \infty$, 
$\gamma$ is an arbitrary multi-indices $(\gamma_1, \gamma_2, \ldots, \gamma_d) \in (\mathbb{Z}_{\geq 0})^d$ 
and $|\gamma| = \gamma_1 + \gamma_2 + \cdots + \gamma_d$.
When $\Omega$ is a smooth bounded domain, the basic properties of analytic semigroups 
together with the standard $L^p$ elliptic estimates 
imply that the same type of estimate holds for $1 < p < \infty$ and multi-indices $\gamma$ satisfying $|\gamma| = 0, 1, 2$.
One of our main theorem provides higher-order derivative estimates on bounded domains which includes the endpoint cases $p = 1$ and $p = \infty$.

\vspace{3mm}

\begin{thm} \label{thm:1}
Let $d \geq 2$, $\gamma \in (\mathbb{Z}_{\geq 0})^d$ and $1 \leq p \leq \infty$. 
Suppose that $\Omega$ is a bounded domain of ${\mathbb R}^d$ with $C^{|\gamma| + 2}$ boundary. 
Then, a positive constant $C$ exists such that 
for every $u_0 \in C_c^\infty (\Omega)$, $u(t) := e^{t \mathcal{L}} u_0$ in $\mathcal{X}'(\mathcal{L})$ satisfies that
\begin{equation} \label{eq:20260323-3}
\begin{aligned} 
\|  \partial_x^\gamma u(t) \|_{L^p} &\leq C t^{-|\gamma| / 2} \| u_0 \|_{L^p}
\quad \textup{ for all } t \in (0, 1).
\end{aligned}
\end{equation}
For $t \geq 1$, the norm decays exponentially.
\end{thm}

\vspace{3mm}

The second estimate concerns the equivalence between the $L_t^q L_x^p$ norm of solutions 
and the Besov norm of the initial data.
By the characterization of homogeneous Besov spaces via the heat kernel, 
it is well known that for $1 \le p \le \infty$ and $1 \le q < \infty$, 
the $L_t^q L_x^p$ norm of solutions to the linear heat equation on the whole space is equivalent to 
the $\dot{B}_{p,q}^{-2/q}$ norm of the initial data (e.g. Section~2.4 in the book~\cite{book:BCD}, Section~5.7 in the book~\cite{book:LePi}). 
\begin{equation*}
C^{-1} \| f \|_{\dot{B}^{-\frac{2}{q}}_{p, q}(\mathbb{R}^d)}
\leq \| e^{t \Delta} f \|_{L_t^q(0, \infty; L_x^p(\mathbb{R}^d))} 
\leq C \| f \|_{\dot{B}^{-\frac{2}{q}}_{p, q}(\mathbb{R}^d)}.
\end{equation*}
As a related result in the whole space, $L^1$-maximal regularity for more general parabolic equations, 
including the parabolic Lam\'{e} system, has been established in~\cite{paper:OgSh-2016}. 
For bounded domains, 
the following inequality follows from the $L^1$-maximal regularity results established 
in~\cite{paper:DaTo-2022} and~\cite{paper:KuSh-2025}.
\begin{equation*}
\| u \|_{L_t^1(0, T; B^{s+2}_{p, 1}(\Omega))} \leq C \| u_0 \|_{B^s_{p, 1}(\Omega)}
\quad (1 < p < \infty, -1 + \frac{1}{p} < s < \frac{1}{p}), 
\end{equation*}
where the space $B^s_{p, 1}(\Omega)$ is a Besov space on the domain, 
defined as the restriction of functions of the whole space.
Moreover, \cite{paper:ChOgSh-2025} proved $L^1$-maximal regularity for the parabolic Lam\'{e} system on the half-space. 

We introduce Besov space on a domain.

\begin{dfn}
For $s \in \mathbb{R}$ and $1 \leq p, q \leq \infty$, $\dot{B}_{p, q}^s(\mathcal{L})$ is defined by letting
\begin{equation*}
\begin{aligned}
&\dot{B}_{p, q}^s(\mathcal{L}) \coloneqq \{ f \in \mathcal{Z}'(\mathcal{L}) 
\mid \| f \|_{\dot{B}_{p, q}^s(\mathcal{L})} < \infty \}, 
\\
&\| f \|_{\dot{B}_{p, q}^s(\mathcal{L})} \coloneqq 
\| \{2^{sj} \| \phi_j(\sqrt{-\mathcal{L}}) f \|_{L^p(\Omega)} \}_{j \in \mathbb{Z}} \|_{\ell^q(\mathbb{Z})}. 
\end{aligned}
\end{equation*}
\end{dfn}

As a similar estimate on bounded domains, 
we establish the following equivalence with respect to the homogeneous Besov norm defined above.

\begin{thm} \label{thm:2}
Let $d \geq 2$, $s \in \mathbb{R}, s_0 > s/2, 1 \leq p, q \leq \infty$. 
Then, a positive constant $C$ exists such that for every $u_0 \in \dot{B}^{s}_{p, q}(\mathcal{L})$, 
\begin{equation} \label{eq:20260205-1} 
\begin{aligned} 
C^{-1} \| u_0 \|_{\dot{B}^{s}_{p, q}(\mathcal{L})} 
\leq \left\{ \int_0^\infty (t^{-\frac{s}{2}} \| (-t \mathcal{L})^{s_0} e^{t \mathcal{L}} u_0 \|_{L^p(\Omega)} )^q
\frac{\mathrm{d}t}{t} \right\}^{\frac{1}{q}}
\leq C \| u_0 \|_{\dot{B}^{s}_{p, q}(\mathcal{L})}.
\end{aligned}
\end{equation}
Moreover, when $\alpha \in \mathbb{Z}_{\geq 0}, \alpha / 2 + s_0 > s / 2$, the following holds.
\begin{equation} \label{eq:20260205-2} 
\begin{aligned}
C^{-1} \| u_0 \|_{\dot{B}^{s}_{p, q}(\mathcal{L})} 
\leq \left\{ \int_0^\infty (t^{-\frac{s}{2}} \| t^{\frac{\alpha}{2}} \nabla^\alpha (-t \mathcal{L})^{s_0} e^{t \mathcal{L}} u_0 \|_{L^p(\Omega)} )^q
\frac{\mathrm{d}t}{t} \right\}^{\frac{1}{q}}
\leq C \| u_0 \|_{\dot{B}^{s}_{p, q}(\mathcal{L})}.
\end{aligned}
\end{equation}
\end{thm}

\begin{rem*}
For $1 \le q < \infty$, $s = -2/q$, and $s_0 = 0$, 
the middle term in the above expression coincides with $\|u\|_{L_t^q(0,\infty; L_x^p(\Omega))}$. 
\end{rem*}

\begin{rem*}
Compared with~\cite{paper:DaTo-2022} and~\cite{paper:KuSh-2025}, 
the significance of estimate~\eqref{eq:20260205-1} lies in the fact that it includes the spatial Lebesgue exponents $p = 1, \infty$.
Moreover, a key feature of the estimate~\eqref{eq:20260205-2} is that it takes into account derivatives in all directions, 
rather than only those represented by $(-\mathcal{L})^{s_0}$.
\end{rem*}

We give a few comments on our proof.
Theorem~\ref{thm:1} is proved by a method similar to that used in \cite{paper:FuIw-2025}.
For $p = 1$, we use spectral multiplier and modify the method in~\cite{paper:IwMaTa-2018} to the higher-order resolvent of Lam\'{e} operator. 
As for $p = \infty$, we argue by contradiction used in~\cite{paper:AbGi-2013} for the Stokes equation. 
We extend the argument for higher order derivatives, introducing the following quantity
\[
N_\ell[u](t, x) := \sum_{s = 0}^{\ell} t^{s/2} |\nabla^s u(t, x)|.
\]
We show Theorem~\ref{thm:2} in a method analogous to that in the case when $\Omega = \mathbb{R}^d$ 
together with estimates in Theorem~\ref{thm:1} and the boundedness of the spectral multiplier.

This paper is organized as follows.
In Section~2, we prepare some propositions to prove the Theorem~\ref{thm:1}.
In Sections~3 and~4, we prove the Theorem~\ref{thm:1} for $p = 1, \infty$. 
In Section 5, we prove Theorem~\ref{thm:2}. 

\vskip3mm


\section{Preliminaries}

We prepare some propositions for the proof of Theorem~\ref{thm:1} and Theorem~\ref{thm:2}.

First, we present higher-order elliptic estimates for $2 \le p < \infty$.
This proposition will be used repeatedly in the proof of the main result.

\begin{prop}[{\cite[Proposition~3.3]{paper:DaTo-2022}}] \label{prop:200}
Let $k \in \mathbb{N} \cup \{0\}, p \in [2, \infty)$ and 
let $\Omega \subset \mathbb{R}^d$ be a bounded domain with $C^{k+2}$ boundary.
Assume $u \in W^{2, p}(\Omega), f \in W^{k, p}(\Omega)$ satisfies the following equation.
\begin{equation} \label{eq:EL}
\begin{cases}
(\mu \Delta + (\mu + \lambda)\nabla \text{div}) u = f
&\quad \text{in } \ \Omega, \\
u =  0
&\quad \text{on } \ \partial \Omega.
\end{cases}
\end{equation}
Then, $u \in W^{k+2, p}(\Omega)$ and there exists a constant $C > 0$ satisfying
\begin{equation*}
\| u \|_{W^{k+2, p}(\Omega)} \leq C (\| u \|_{L^p(\Omega)} + \| f \|_{W^{\ell, p}(\Omega)} ).
\end{equation*}
\end{prop}

\subsection{Proposition for the proof of Theorem~\ref{thm:1} when $p = 1$}

Based on the previous works~\cite{paper:IwMaTa-2018}, \cite{paper:FuIw-2025} for Schr\"odinger operators and the Dirichlet Laplacian, 
we check several propositions.

First, we define higher-order resolvents and state their derivative estimates on $L^2(\Omega)$.

\begin{dfn}
Let $\Omega \subset \mathbb{R}^d$ be a bounded domain, and let $-\mathcal{L}$ denote the Lam\'{e} operator on $L^2(\Omega)$.
Let $M \in \mathbb{N}$ and $t > 0$.
Then, define the higher-order resolvent $R_{M,t} \in \mathcal{B}(L^2(\Omega))$ associated with $-\mathcal{L}$ as follows.
\begin{equation*}
R_{M, t} \coloneqq (1 - t \mathcal{L})^{-M}.
\end{equation*}
\end{dfn}

\begin{lem} \label{lem:12-013}
Let $\Omega \subset \mathbb{R}^d$ be a bounded domain and let $M \in \mathbb{N}$.
Then, there exists $C > 0$ such that
\begin{equation*}
\begin{aligned}
&\| R_{M, t} \|_{\mathcal{B}(L^2(\Omega))} \leq C, 
\\
&\| \nabla R_{M, t} \|_{\mathcal{B}(L^2(\Omega))} \leq C t^{-1/2}, 
\end{aligned}
\end{equation*}
for $t > 0$.
\end{lem}

\begin{proof}[Proof of Lemma~\ref{lem:12-013}]
Since an estimate for $R_{M,t}$ itself is obvious from its definition, 
we focus on obtaining derivative estimates.
For any $f \in L^2(\Omega)$, we have $R_{M,t}f \in \mathcal{D}(\mathcal{L})$, 
and hence, by the spectral theorem, the following holds for any $t > 0$.
\begin{equation*}
\begin{aligned}
&\quad \| \nabla R_{M, t} f \|_{L^2(\Omega)}^2
\\
&\leq C \< -\mathcal{L} R_{M, t} f, R_{M, t} f \>_{L^2(\Omega)}
\\
&= C \int_0^\infty \dfrac{\lambda}{(1 + t \lambda)^{2M}} \, \textrm{d}\| E_{-\mathcal{L}}(\lambda) f \|_{L^2(\Omega)}^2
\\
&\leq C \int_0^\infty t^{-1} \cdot \dfrac{t \lambda}{1 + t \lambda} \cdot \dfrac{1}{(1 + t \lambda)^{2M-1}} 
\, \textrm{d}\| E_{-\mathcal{L}}(\lambda) f \|_{L^2(\Omega)}^2
\\
&\leq C t^{-1} \int_0^\infty \, \textrm{d}\| E_{-\mathcal{L}} f \|_{L^2(\Omega)}^2
\\
&= C t^{-1} \| f \|_{L_2(\Omega)}^2.
\end{aligned}
\end{equation*}
Here, the first inequality follows from the coercivity property (e.g. \cite{paper:MiMo-2010}).
\begin{equation}
\begin{aligned}
&\quad _{H^{-1}}\< -\mathcal{L} u, u \>_{H_0^1}
\\
&= \mu \int_\Omega \text{rot} \, u \cdot \text{rot} \, u \, \textrm{d}x
  + (2 \mu + \lambda) \int_\Omega \text{div} \, u \, \text{div} \, u \, \textrm{d}x
\\
&\geq \min{(\mu, 2 \mu + \lambda)} \int_\Omega (|\text{rot} \, u|^2 + |\text{div} \, u|^2) \, \textrm{d} x
\\
&= \min{(\mu, 2 \mu + \lambda)} \int_\Omega |\nabla u|^2 \, \textrm{d} x.
\end{aligned}
\end{equation}
\end{proof}

Before establishing higher-order derivative estimates on $L^2(\Omega)$, 
we first recall basic estimate without derivatives.

For $\varphi \in \mathcal{S}(\mathbb{R})$, define $\varphi(-t \mathcal{L})$ by spectral theorem.
Then, we have a classical result that
\begin{equation*}
\| \varphi(-t \mathcal{L}) \|_{\mathcal{B}(L^2(\Omega))} \leq C, 
\quad t > 0.
\end{equation*}
In fact, for any $f \in L^2(\Omega)$ and $t > 0$, it follows
\begin{equation*}
\begin{aligned}
\| \varphi(-t \mathcal{L}) \|_{L^2(\Omega)}^2
= \int_0^\infty | \varphi(t \lambda) |^2 \, \mathrm{d} \| E_{-\mathcal{L}}(\lambda) f \|_{L^2(\Omega)}^2
\leq \| \varphi \|_{L^\infty(\Omega)}^2 \| f \|_{L^2(\Omega)}^2.
\end{aligned}
\end{equation*}

By combining the above $L^2$--$L^2$ estimate with higher-order elliptic estimates (Proposition~\ref{prop:200}), 
we obtain estimates for derivatives of higher order.

\begin{prop} \label{prop:12-010}
Let $-\mathcal{L}$ denote the Lam\'{e} operator on $L^2(\Omega)$ and $\varphi \in \mathcal{S}(\mathbb{R})$.
Then, for any $M, k \in \mathbb{Z}_{\geq 0}$ satisfing $0 \leq k \leq 2M$, there exists $C > 0$ such that
\begin{equation*}
\begin{aligned}
\| \nabla^k R_{M, t} f \|_{L^2(\Omega)} \leq C t^{-k/2} \| f \|_{L^2(\Omega)}, 
\\
\| \nabla^k \varphi(-t \mathcal{L}) f \|_{L^2(\Omega)} \leq C t^{-k/2} \| f \|_{L^2(\Omega)}, 
\end{aligned}
\end{equation*}
for any $f \in L^2(\Omega)$ and $t \in (0, 1)$. 
\end{prop}

The proof follows the same procedure as in~\cite{paper:FuIw-2025}.

Next, let us define scaled amalgam spaces on $\Omega$ as follows (see~\cite[Section~4]{paper:IwMaTa-2018}).

\begin{dfn} \label{def:12-002}
Let $1 \leq p, q \leq \infty$ and $t > 0$. Define the space $\ell^p(L^q)_{t}$ by
\begin{equation*}
\ell^p(L^q)_{t} = \ell^p(L^q)_{t}(\Omega)
\coloneqq \left\{ f \in L^q_{\mathrm{loc}}(\overline{\Omega}) \, \mid \,
\| f \|_{\ell^p(L^q)_{t}} < \infty \right\}, 
\end{equation*}
\begin{equation*}
\| f \|_{\ell^p(L^q)_{t}}
\coloneqq
\begin{cases}
\displaystyle  \left( \sum_{n \in \mathbb{Z}^d} \| f \|^p_{L^q(C_{t}(n))} \right)^{\frac{1}{p}}
& \quad \text{for $1 \leq p< \infty $,}
\\
\displaystyle \sup_{n \in \mathbb{Z}^d} \| f \|_{L^q(C_{t}(n))}
& \quad \text{for $ p=\infty $}.
\end{cases}
\end{equation*}
Here, $C_t(n)$ is the hyper cube centered at $t^{1/2} n \in t^{1/2} \mathbb{Z}^d$ with side length $t^{1/2}$:
\begin{equation*}
C_{t}(n) \coloneqq \left\{ x \in \Omega \, \mid \, 
\max_{j = 1, \cdots, d} |x_j - t^{1/2} n_j| \leq \frac{t^{1/2}}{2} \right\}.
\end{equation*}
\end{dfn}

Next, we provide $L^p$--$\ell^p(L^q)_t$ estimates for the resolvent of $-\mathcal{L}$.

\begin{prop} \label{prop:207}
Let $1 \leq p \leq q \leq \infty$, and $\beta$ such that
\begin{equation*}
\beta > \dfrac{d}{2} \left( \dfrac{1}{p} - \dfrac{1}{q} \right)
\end{equation*}
Let $-\mathcal{L}$ denote the Lam\'{e} operator on $L^2(\Omega)$.
Then, $(1 - t \mathcal{L})^{-\beta}$ is extended to a bounded linear operator 
from $L^p(\Omega)$ to $\ell^p(L^q)_t$ and the following holds.
\begin{equation*}
\| (1 + t(-\mathcal{L}))^{-\beta} \|_{\mathcal{B}(L^p(\Omega), \ell^p(L^q)_t)}
\leq C t^{-(d/2)(1/p - 1/q)}
\end{equation*}
for any $0 < t \leq 1$.
\end{prop}

As in~\cite[Proposition~4.1]{paper:IwMaTa-2018}, the proof of Proposition~\ref{prop:207} can be carried out 
by combining the integral representation of the resolvent with integral kernel estimates in scaled amalgam spaces.
However, unlike the case of Laplacian, a simple zero extension does not allow us to estimate the integral kernel on a bounded domain.
Therefore, we mention the following lemma, which provides pointwise estimates for the integral kernel.

\begin{lem} \label{lemma:021}
Let $\Omega$ be a smooth bounded domain of $\mathbb{R}^d, d \geq 1$.
Then, the semigroup $e^{t\mathcal{L}}$ on $L^2(\Omega)$ is given by an integral kernel $p(t, x, y)$ 
which satisfies the Gaussian upper bound.
\begin{equation*}
|p(t, x, y)| \leq C t^{-d/2} \exp{\left( -\dfrac{|x-y|^2}{ct} \right)} 
\quad \text{ for } \, t > 0, \textup{a.e.} \, (x, y) \in \Omega \times \Omega, 
\end{equation*}
where $c > 0$ is a constant.
\end{lem}

\begin{pr}
The proof is similar to the argument as in~\cite{book:Ou}
(see Section~6.3 in~\cite{book:Ou} for more details).
\end{pr}


Finally, the following proposition provides several properties concerning commutator estimates.

\begin{dfn} \label{def:12-005}
Let $x \in \mathbb{R}^d, n \in \mathbb{Z}^d$ and $t > 0$.
Let $k \in \mathbb{Z}_{\geq 0}$.
For an operator $L$ on $L^2(\Omega)$, we define the commutator $\operatorname{Ad}^k(L)$ as follows.
\begin{equation*}
\begin{aligned}
&\operatorname{Ad}^0(L) = L, 
\\
&\operatorname{Ad}^k(L) = \operatorname{Ad}^{k-1}( (x_j - t^{1/2} n_j) L - L (x_j - t^{1/2} n_j) ), 
\quad k \geq 1, 
\end{aligned}
\end{equation*}
where $x_j, n_j$ denote the $j$-th element of $x, n$ respectively.
\end{dfn}

\begin{prop}[{\cite{paper:JeNa-1995}}] \label{prop:12-004}
Let $M \in \mathbb{N}, t > 0$ and let $-\mathcal{L}$ denote the Lam\'{e} operator on $L^2(\Omega)$.
Then, for higher order resolvent $R_{M, t} = (1 - t \mathcal{L})^{-M}$, 
a sequence of operators $\{ \operatorname{Ad}^{\ell}(R_{M, t}) \}_{\ell \geq 0}$ satisfies the following.
\begin{equation*}
\begin{aligned}
&\operatorname{Ad}^0(R_{M, t}) = R_{M, t}, 
\\
&\operatorname{Ad}^\ell ( R_{M, t} )
= \sum_{k = 0}^{\ell-1} \sum_{m_1 + m_2 = M + 1}
\binom{\ell}{k} \operatorname{Ad}^k ( R_{m_1, t} ) \operatorname{Ad}^{\ell-k} ( - t \mathcal{L} ) R_{m_2, t}
, \quad \ell \geq 1, 
\end{aligned}
\end{equation*}
where $m_1$ and $m_2$ in the above sum satisﬁes $m_1, m_2 \geq 1$.

Moreover, for $k \in \mathbb{Z}_{\geq 0}$, it holds
\begin{equation*}
\begin{aligned}
&\quad \mathrm{Ad}^{k+1}(e^{-itR_{M, t}}) 
\\
&= - i \int^{\tau}_{0} \sum_{k_1 + k_2 + k_3 = k}
\frac{k!}{k_1! k_2! k_3!} \mathrm{Ad}^{k_1}(e^{-isR_{M, t}}) \mathrm{Ad}^{k_2+1}(R_{M, t}) \mathrm{Ad}^{k_3}(e^{-i(\tau-s)R_{M, t}}) \, {\mathrm d}s.
\end{aligned}
\end{equation*}
where $k_1, k_2$ and $k_3$ in the above sum satisﬁes $k_1, k_2, k_3 \geq 0$.
\end{prop}

\subsection{Proposition for the proof of Theorem~\ref{thm:1} when $p = \infty$}

We consider the estimate for the following quantity.
\begin{equation} \label{eq:20260323-5}
N_\ell[u](t, x) \coloneqq \sum_{s = 0}^{\ell} t^{s/2} |\nabla^s u(t, x)|.
\end{equation}

We give an $L^\infty$ estimate of $N_{\ell}[u]$, which is known for the Stokes equations for $\ell = 2$ (see~\cite{paper:AbGi-2013}).
\begin{lem} \label{lem:220}
Let $\Omega$ be a bounded domain with $C^{\ell + 1} boundary$.
Assume that $u_0 \in C_c^\infty(\Omega)$, and define $u(t) \coloneqq e^{t \mathcal{L}} u_0$.
Then, 
\begin{equation*}
\sup_{0 < t < 1}{\| N_{\ell}[u](t) \|_{L^\infty} }< \infty.
\end{equation*}
\end{lem}

\begin{proof}
Let $d < r < \infty$. 
We will show that there exists $C = C(\Omega, r) > 0$ such that
\begin{equation} \label{eq:G}
\begin{aligned}
\sup_{0 < t < 1}{ \sum_{s = 0}^{\ell} t^{s/2} \| \nabla^s u(t, x) \|_{W^{1, r}} }
\leq C ( \| u_0 \|_{L^r}, \| \mathcal{L} u_0 \|_{L^r}, \ldots, \| \mathcal{L}^{\lceil (\ell+2)/2 \rceil} u_0 \|_{L^r} ),
\end{aligned}
\end{equation}
where $\lceil \alpha \rceil$ is the smallest integer greater than or equals to $\alpha$, for $\alpha \in \mathbb{R}$.

By higher order elliptic estimates (Proposition~\ref{prop:200}) and $L^r$ boundedness of semigroup, 
we obtain for even number $s = 2j, j \in \mathbb{Z}$, 
\begin{equation*}
\begin{aligned}
\| \nabla^s u(t, x) \|_{L^r}
&\leq C ( \| u \|_{L^r} + \| \mathcal{L} u \|_{W^{s-2, r}} )
\\
&\leq C (\| u \|_{L^r} + \| \mathcal{L} u \|_{L^r} + \| \mathcal{L}^2 u \|_{W^{s-4, r}} )
\\
&\vdots
\\
&\leq C ( \| u \|_{L^r} + \cdots + \| \mathcal{L}^j u \|_{L^r} )
\\
&\leq C (\| u_0 \|_{L^r} + \cdots + \| \mathcal{L}^j u_0 \|_{L^r} ).
\end{aligned}
\end{equation*}
Similarly, when $s = 2j+1, j \in \mathbb{Z}$ is odd, the following estimate holds.
\begin{equation*}
\begin{aligned}
\| \nabla^s u(t, x) \|_{L^r}
\leq C (\| u_0 \|_{L^r} + \cdots + \| \mathcal{L}^{j+1} u_0 \|_{L^r} ).
\end{aligned}
\end{equation*}
Therefore, we prove \eqref{eq:G}.
\end{proof}

We introduce the following proposition on the uniqueness for the parabolic Lam\'{e} system on the whole space and half space, 
which is used in the proof of Theorem~\ref{thm:1}.
The proof is similar to that for the heat equation on the whole space in~\cite{book:GiGiSa}.

\begin{prop} \label{prop:12-015}
\noindent(i)
Let $\Omega = \mathbb{R}^d$. 
Suppose that $u = u(t, x) \in L_{\text{loc}}^1((0, T) \times \Omega)$ satisfies the following.
\begin{equation} \tag{a}
\sup_{0 < t < T} \| u(t) \|_{L^\infty} < \infty
\end{equation}
and
\begin{equation} \tag{b}
\int_0^T \int_{\Omega} u \cdot (\partial_t \phi + \mathcal{L} \phi) {\mathrm d}x {\mathrm d}t = 0
\end{equation}
for any $C_c^\infty([0, T) \times \mathbb{R}^d)$. 
Then, $u$ is $0$ almost everywhere.

\vspace{3mm}

\noindent(ii)
Let $\Omega = \mathbb{R}^d_+$. 
Suppose that $u = u(t, x) \in L_{\text{loc}}^1((0, T) \times \Omega)$ satisfies the following.
\begin{equation*}
\int_0^T \int_{\Omega} u \cdot (\partial_t \phi + \mathcal{L} \phi) {\mathrm d}x {\mathrm d}t = 0
\end{equation*}
for any
\begin{equation*}
\phi \in C([0, T] ; W^{1,1}_0(\Omega) \cap W^{2, 1}(\Omega) \cap C(\Omega)) 
\cap C^1([0, T] ; L^1(\Omega))
\end{equation*}
with $| \nabla \phi(x)| \leq C(1 + |x|^2)^{-d/2}$.
Then, $u$ is $0$ almost everywhere.
\end{prop}

\section{$L^1$ estimate in Theorem~\ref{thm:1}} 

The result can be proved by arguments similar to those in~\cite{paper:FuIw-2025}, 
which established higher-order derivative estimates for the linear heat equation, 
together with Propositions~\ref{prop:200}--\ref{prop:12-004} (see also~\cite{paper:JeNa-1994}, \cite{paper:JeNa-1995} and~\cite{paper:IwMaTa-2018}).

To clarify the differences, we focus here on the following estimate.
\begin{prop}
For $\ell \in \mathbb{N}, \alpha > d / 2$ and $\widetilde \varphi \in \mathcal{S}(\mathbb{R})$, it holds
\begin{equation*}
||| \nabla^\ell \widetilde \varphi(- t \mathcal{L}) |||_\alpha \leq C t^{\frac{\alpha - \ell}{2}}, 
\quad 0 < t < 1, 
\end{equation*}
where $||| \nabla^\ell \widetilde \varphi(- t \mathcal{L}) |||_{\alpha}$ is defined by
\[
||| \nabla^\ell \widetilde \varphi(- t \mathcal{L}) |||_\alpha
\coloneqq \sup_{n \in {\mathbb Z}^d} 
  \left\| | \cdot - t^{1/2} n |^\alpha \nabla ^\ell \widetilde \varphi(- t \mathcal{L}) \chi_{C_t(n)}  \right\|_{L^2 \to L^2}.
\]
\end{prop}

\begin{proof}
We follow the proof of Theorem~1.1 for $L^1$ estimate in~\cite{paper:FuIw-2025} 
(see also the proof of Lemma~7.1 in \cite{paper:IwMaTa-2018}).
We aim to show that
\begin{equation} \label{eq:12-006}
\left\| | \cdot - t^{1/2} n |^\alpha \nabla ^\ell R_{M, t} e^{i \tau R_{M, t}} \chi_{C_t(n)}  \right\|_{L^2 \to L^2} 
\leq C (1 + |\tau|^\alpha) t^{\frac{\alpha}{2} - \frac{\ell}{2}}
\end{equation}
for the case when $\alpha = 1$.

We define the commutator $\operatorname{Ad}^k$ as in Definition~\ref{def:12-005}:
\begin{equation*}
\begin{aligned}
&\operatorname{Ad}^0(L) = L, 
\\
&\operatorname{Ad}^k(L) = \operatorname{Ad}^{k-1}( (x_j - t^{1/2} n_j) L - L (x_j - t^{1/2} n_j) ), 
\quad k \geq 1, 
\end{aligned}
\end{equation*}
where $x_j, n_j$ denote the $j$-th element of $x, n$ respectively.
We can check that, the following holds.
\begin{equation} \label{eq:12-007}
\operatorname{Ad}^k (-t \mathcal{L}) f = 
\begin{cases}
-t \mathcal{L} f & \text{ if } k = 0, \\
2 \mu t (\partial_j f_j) e_j + (\mu + \lambda)t (\nabla f_j + (\text{div} \, f) e_j) & \text{ if } k = 1, \\
(-2\mu t - 2(\mu + \lambda t)) f_j e_j & \text{ if } k = 2, \\
0 & \text{ if } k \geq 3, \\
\end{cases}
\end{equation}
for $f \in L^2(\Omega)$. 
Here, $e_j$ denotes the unit vector whose $j$-th component is equal to $1$.

We return to the proof of~\eqref{eq:12-006}.
By commuting $(x_j - t^{1/2} n_j)$ with the operators $\nabla$, $R_{M, t}$, and $e^{i \tau R_{M, t}}$, 
the following holds for $\ell \ge 1$ and $f \in L^2(\Omega)$.
\[
\begin{aligned}
&\quad \left\| ( x_j - t^{1/2} n_j ) \nabla ^\ell R_{M, t} e^{i \tau R_{M, t}} \chi_{C_t(n)} f \right\|_{L^2}
\\
&\leq \ell \| \nabla^{\ell - 1} R_{M, t} e^{i \tau R_{M, t}} \chi_{C_t(n)} f \|_{L^2} + 
  \| \nabla^\ell \mathrm{Ad}^1 (R_{M, t}) e^{i \tau R_{M, t}} \chi_{C_t(n)} f \|_{L^2} 
  \\
  &\quad + \| \nabla^\ell R_{M, t} \mathrm{Ad}^1 (e^{i \tau R_{M, t}}) \chi_{C_t(n)} f \|_{L^2} + 
  \| \nabla^\ell R_{M, t} e^{i \tau R_{M, t}} (x_j - t^{1/2} n_j)  \chi_{C_t(n)} f \|_{L^2}
\\
&\eqqcolon I + II + III + IV.
\end{aligned}
\]

For $II$, the following holds by Proposition~\ref{prop:12-004} and~\eqref{eq:12-007}.
\begin{equation*}
\begin{aligned}
\mathrm{Ad}^1 ( R_{M, t} )
= \sum_{m_1 + m_2 = M+1} (1 - t \mathcal{L})^{-m_1} \mathrm{Ad}^1(- t \mathcal{L}) (1 - t \mathcal{L})^{-m_2}, 
\end{aligned}
\end{equation*}
where $m_1$ and $m_2$ in the above sum satisfies $m_1, m_2 \geq 1$.

In addition, we can prove $\ell-$th derivatives estimate for any $1 \leq i, j, k \leq d$, $m_1, m_2 \geq 0$ 
and $\ell \leq 2(m_1 + m_2) - 1$, 
\begin{equation} \label{eq:12-011}
\| \nabla^\ell (1-t \mathcal{L})^{-m_1} e_k (t \partial_{x_j}) ((1-t \mathcal{L})^{-m_2} f)_i \|_{L^2}
\leq C t^{-\frac{\ell - 1}{2}} \| f \|_{L^2}, t \in (0, 1), 
\end{equation}
by using the following estimate and induction on $\ell$. 

In fact, the case when $\ell = 0, 1$ follows from Proposition~\ref{prop:12-010}. 
If we assume the inequalities up to $(\ell-1)$-th derivative order, 
then it follows from higher order elliptic estimates and Proposition~\ref{prop:12-010}.
\begin{equation*}
\begin{aligned}
\| \nabla^\ell (1-t \mathcal{L})^{-m_1} e_k (t \partial_{x_j}) ((1-t \mathcal{L})^{-m_2} f)_i \|_{L^2}
\leq C t^{-\frac{\ell-1}{2}} \| f \|_{L^2}, 
\end{aligned}
\end{equation*}
for $0 < t < 1$ and $f \in L^2(\Omega)$.
This proves the case of $\ell$-th order and implies
\begin{equation*}
\| \nabla^\ell \mathrm{Ad}^1 (R_{M, t}) \|_{L^2 \to L^2} \leq C t^{-\frac{\ell-1}{2}}.
\end{equation*}

Next, we see the estimate of $\mathrm{Ad}^1 (e^{i \tau R_{M, t}})$ appearing in $III$.
From the latter part of Proposition~\ref{prop:12-010}, we have
\begin{equation*}
\begin{aligned}
\mathrm{Ad}^1 (e^{i \tau R_{M, t}}) = 
- i \int^{\tau}_{0} (e^{-isR_{M, t}}) \mathrm{Ad}^{1}(R_{M, t}) (e^{-i(\tau-s)R_{M, t}}) \, {\mathrm d}s
\end{aligned}
\end{equation*}
and
\begin{equation*}
\begin{aligned}
&\quad \| \mathrm{Ad}^1 (e^{i \tau R_{M, t}}) \|_{L^2 \to L^2}
\\
&\leq \int_{0}^{\tau} \| e^{-isR_{M, t}} \|_{L^2 \to L^2} 
  \| \mathrm{Ad}^{1}(R_{M, t}) \|_{L^2 \to L^2} \| e^{-i(\tau-s)R_{M, t}} \|_{L^2 \to L^2} \, {\mathrm d}s
\\
&\leq C t^{1/2} | \tau |.
\end{aligned}
\end{equation*}

In conclusion, from $I$ to $IV$, we obtain estimates analogous to those in~\cite{paper:FuIw-2025}.
\begin{equation*}
\begin{aligned}
&I \leq C \| \nabla^{\ell - 1} R_{M, t} \|_{L^2 \to L^2}  \| e^{i \tau R_{M, t}} \|_{L^2 \to L^2}  \| \chi_{C_t(n)} f \|_{L^2} \leq C t^{-\frac{\ell-1}{2}} \| f \|_{L^2}, 
\\
&II \leq C \| \nabla^\ell \mathrm{Ad}^1 (R_{M, t}) \|_{L^2 \to L^2} \| e^{i \tau R_{M, t}} \|_{L^2 \to L^2}  \| \chi_{C_t(n)} f \|_{L^2} \leq C t^{-\frac{\ell-1}{2}} \| f \|_{L^2}, 
\\
&III \leq \| \nabla^\ell R_{M, t} \|_{L^2 \to L^2}  \| \mathrm{Ad}^1 (e^{i \tau R_{M, t}}) \|_{L^2 \to L^2}  \| \chi_{C_t(n)} f \|_{L^2} \leq C t^{-\frac{\ell-1}{2}} |\tau| \| f \|_{L^2}, 
\\
&IV \leq \| \nabla^\ell R_{M, t} \|_{L^2 \to L^2}  \| e^{i \tau R_{M, t}} \|_{L^2 \to L^2}  \| (x_j - t^{1/2} n_j)  \chi_{C_t(n)} f \|_{L^2} \leq C t^{-\frac{\ell-1}{2}} \| f \|_{L^2}.
\end{aligned}
\end{equation*}
Therefore, \eqref{eq:12-006} is proved.

\end{proof}

\section{$L^\infty$ estimate in Theorem~\ref{thm:1}}

We prove the estimate for $p = \infty$ in Theorem~\ref{thm:1}, we will show the following.

\begin{thm}
Let $\ell \geq 2$. Suppose that $\Omega \subset {\mathbb R}^d$ be a bounded domain with $C^{\ell+2}$ boundary.
Then, there exists $C > 0$ and $T > 0$ such that
\[
\sup_{0 < t < T}{\| N_\ell[u](t) \|_{L^\infty}} < C \| u_0 \|_{L^\infty}
\]
holds for any solution $u(t)$ of $\eqref{eq:L}$ with $u_0 \in C_c^{\infty}(\Omega)$, 
where $ N_\ell[u]$ is defined by~\eqref{eq:20260323-5}.
\end{thm}

\begin{pr}
We follow the arguments in~\cite{paper:AbGi-2013} for the Stokes equations (see~\cite{paper:FuIw-2025}).
We prove by contradiction. Assume that there exist a sequence $\{ u_m \}_{m \in {\mathbb N}}$ 
of solution to $\eqref{eq:L}$ and a sequence $\{ \tau_m \}_{m \in {\mathbb N}}$ such that
\[
\begin{aligned}
    \| N_\ell[u](\tau_m) \|_{L^\infty} 
    > m \| u_{0, m} \|_{L^\infty}, 
    \\
    \tau_m \searrow 0 \, (m \to \infty),
\end{aligned}
\]
where $u_{0, m} \in C_c^{\infty}(\Omega)$ is an initial data of $u_m$.
  
By Lemma~\ref{lem:220}, it holds
\[
M_m \coloneqq \sup_{t \in (0, \tau_m)}{\| N[u_m](t) \|_{L^\infty}} < \infty,
\]
and we define $v_m$ as follows:
\begin{equation*}
\begin{aligned}
&\quad v_m (t, x) \coloneqq \frac{ u_m (t_m t, x_m + t_m^{1/2} x) }{M_m} 
\quad \text{ for } (t, x) \in (0, 1) \times \Omega_m, 
\\
&\text{where } \Omega_m := \Big\{ y \in \mathbb{R}^d \, \Big| \, y = \frac{x - x_m}{t_m^{1/2}}, x \in \Omega \Big\}.
\end{aligned}
\end{equation*}
Here, $(t_m, x_m) \in (0, \tau_m) \times \Omega$ is a point satisfying $N_\ell[u_m](t_m, x_m) \geq M_m / 2$.
Then, $v_m$ satisfies the parabolic Lam\'{e} system and the following.
\begin{equation*} \label{eq:12-014}
\begin{aligned}
N_\ell[v_m](1, 0) \geq \frac{1}{2}, 
\quad \sup_{t \in (0, 1)} \| N_\ell[v_m](t) \|_{L^\infty} \leq 1, 
\quad \| v_{0, m} \|_{L^\infty} < \frac{1}{m}.
\end{aligned}
\end{equation*}

Let us define the quantity 
\[
c_m \coloneqq \frac{{\rm dist}(\Omega, x_m)}{t_m^{1/2}} = {\rm dist}(\Omega_m, 0), 
\]
which is related to the limit of $v_m$.

Case 1. $\displaystyle \limsup_{m \to \infty}c_m = \infty$

We can prove samely as Case~1 in the proof of Theorem~4.1 in~\cite{paper:FuIw-2025} with replacing $\Delta$ by $- \mathcal{L}$.

\vspace{3mm}

Case 2. $\displaystyle \limsup_{m \to \infty}c_m < \infty$

Exactly as in~\cite{paper:FuIw-2025}, after proper rotations and translations, 
there exist a sequence $\{v_{m_k}\}_{k \in \mathbb{N}}$ of $\{ v_m \}_{m \in \mathbb{N}}$ 
and a function $v$ satisfying the following properties as Case 1.
\begin{equation*}
\begin{aligned}
v_{m} \to v
\quad \text{ locally uniformly in } \overline{\mathbb{R}}^d_{+, -c_0} \times (0, 1], 
\end{aligned}
\end{equation*}
and local uniform convergence of derivatives up to $\ell$-th order.
Here, 
\[
\mathbb{R}^d_{+, -c_0} \coloneqq \{ (x', x_d) \in \mathbb{R}^d \mid x_d > -c_0 \}.
\]

Fix $R > 0$ and let $B_R^+ = B_R(0, \ldots, 0, -c_0) \cap {\mathbb R}^d_{+, -c_0}$.
Set a test funciton 
\[
\phi \in  C\Big([0, 1] ; W^{1,1}_0(R^d_{+, -c_0}) \cap W^{2, 1}(R^d_{+, -c_0}) \cap C(R^d_{+, -c_0})\Big) \cap C^1\Big([0, 1] ; L^1(R^d_{+, -c_0})\Big)
\]
with $|\nabla \phi(x)| \leq C (1 + |x|^2)^{-d/2}$.
Since $v_{m_k}$ satisfies the parabolic Lam\'{e} system on $\Omega_{m_k} \cap B_R^+$, 
it holds for sufficient large $k$, 
\begin{equation*}
\begin{aligned}
0 
&= \int_0^1 \int_{\Omega_{m_k} \cap B_R^+} \mathcal{L} {v}_{m_k} \cdot \phi \, {\mathrm d}x {\mathrm d}t
- \int_0^1 \int_{\Omega_{m_k} \cap B_R^+} \partial_t {v}_{m_k} \cdot \phi \, {\mathrm d}x {\mathrm d}t
\\
&= - \mu \left( \int_0^1 \int_{\Omega_{m_k} \cap B_R^+} \nabla {v}_{m_k} : \nabla\phi \, {\mathrm d}x {\mathrm d}t
  - \int_0^1 \int_{\partial(\Omega_{m_k} \cap B_R^+)} (\nabla {v}_{m_k} \boldsymbol{n}) \cdot \phi \, {\mathrm d}x {\mathrm d}t \right)
  \\
  &\quad - (\mu + \lambda) \left( \int_0^1 \int_{\Omega_{m_k} \cap B_R^+} \text{div} \, {v}_{m_k} \text{div} \, \phi \, {\mathrm d}x {\mathrm d}t
  - \int_0^1 \int_{\partial(\Omega_{m_k} \cap B_R^+)} \text{div} \, {v}_{m_k} (\phi \cdot \boldsymbol{n}) \, {\mathrm d}x {\mathrm d}t \right)
  \\
  &\quad - \int_0^1 \int_{\Omega_{m_k} \cap B_R^+} \partial_t {v}_{m_k} \cdot \phi \, {\mathrm d}x {\mathrm d}t
\end{aligned}
\end{equation*}

By applying the argument of Theorem~4.1 in~\cite{paper:FuIwKo-2024} to each component, 
we eventually obtain the following.
\begin{equation*}
\begin{aligned}
&\quad \lim_{R \to \infty} \lim_{m \to \infty} \int_0^1 \int_{\Omega_m \cap B_R^+} \mathcal{L} v_{m_k} \cdot \phi \, {\mathrm d}x {\mathrm d}t
= \int_0^1 \int_{\mathbb R^d_{+, -c_0}} {v} \cdot \mathcal{L} \phi \, {\mathrm d}x {\mathrm d}t, 
\\
&\quad \lim_{R \to \infty} \lim_{m \to \infty} \int_0^1 \int_{\Omega_m \cap B_R^+} \partial_t {v}_{m_k} \cdot \phi \, {\mathrm d}x {\mathrm d}t
\\
&= - \int_0^1 \int_{\mathbb R^d_{+, -c_0}} {v} \cdot \partial_t \phi \, {\mathrm d}x {\mathrm d}t
- \int_{\mathbb R^d_{+, -c_0}} {v}(0, x)\cdot \phi(0, x) \, {\mathrm d}x.
\end{aligned}
\end{equation*}
Therefore, ${v}$ satisfies
\[
\int_0^1 \int_{\mathbb R^d_{+, -c_0}} {v} \cdot (\mathcal{L} \phi +\partial_t \phi) \, {\mathrm d}x {\mathrm d}t
= - \int_{\mathbb R^d_{+, -c_0}} {v}(0, x) \cdot \phi(0, x) \, {\mathrm d}x, 
\]
which implies $v \equiv 0$ by Proposition~\ref{prop:12-015}.
However, it contradicts 
\[
N_\ell[v](1, 0) = \lim_{m \to \infty} N_\ell[v_m](1, 0) \geq 1/2.
\]
  
\end{pr}

\section{Proof of Theorem~\ref{thm:2}}

By an argument same as the proof of Theorem~1.3 in~\cite{paper:Iw-2018} with replacing $\Delta$ by $- \mathcal{L}$, 
we can prove~\eqref{eq:20260205-1} and 
\begin{equation} \label{eq:20260322-1} 
\begin{aligned} 
C^{-1} \| u_0 \|_{\dot{B}^{s}_{p, q}(\mathcal{L})} 
\leq \left\{ \int_0^\infty (t^{-\frac{s}{2}} \| (-t \mathcal{L})^{s_0} e^{t \mathcal{L}} u_0 \|_{\dot{B}^{0}_{p, r}(\mathcal{L})} )^q
\frac{\mathrm{d}t}{t} \right\}^{\frac{1}{q}}
\leq C \| u_0 \|_{\dot{B}^{s}_{p, q}(\mathcal{L})}.
\end{aligned}
\end{equation}
with $1 \leq r \leq \infty$. 
We consider the proof of~\eqref{eq:20260205-2}.

\noindent \textit{The first inequality of~\eqref{eq:20260205-2}.}
Note that the operator $\mathcal{L}$ can be written in divergence form
\begin{equation*}
\begin{aligned}
-\mathcal{L} u = -(\mu \Delta + (\mu + \lambda)\nabla \text{div}) u = -\text{div} \, \sigma(u), 
\\
\sigma(u) \coloneqq \mu(\nabla u + (\nabla u)^\top) + \lambda (\text{div} \, u) I.
\end{aligned}
\end{equation*}
We first focus on the case $\alpha = 1$. 
From the inequality~\eqref{eq:20260322-1} with $s_0 = 1/2, s < 1$ and $r = \infty$, we have 
\begin{equation} \label{eq:20260322-2}
\begin{aligned}
\| u_0 \|_{\dot{B}_{p, q}^s}
&\leq C \left\{ \int_0^\infty (t^{-s/2} \| (-t\mathcal{L})^{1/2} e^{t \mathcal{L}} u_0 \|_{\dot{B}_{p, \infty}^0})^q 
\frac{1}{t} \, \mathrm{d}t \right\}^{1/q}
\\
&= C \left\{ \int_0^\infty (t^{-s/2} \| t^{1/2} (-\mathcal{L})^{-1/2} (-\mathcal{L}) e^{t \mathcal{L}} u_0 \|_{\dot{B}_{p, \infty}^0})^q 
\frac{1}{t} \, \mathrm{d}t \right\}^{1/q}
\\
&= C \left\{ \int_0^\infty (t^{-s/2} \| t^{1/2} (-\mathcal{L})^{-1/2} \text{div} \, \sigma( e^{t \mathcal{L}} u_0) \|_{\dot{B}_{p, \infty}^0})^q 
\frac{1}{t} \, \mathrm{d}t \right\}^{1/q}.
\end{aligned}
\end{equation}
Furthermore, by the definition of the Besov norm together with a duality argument, the following holds.
\begin{equation*}
\begin{aligned}
&\quad \| t^{1/2} (-\mathcal{L})^{-1/2} \text{div} \, \sigma( e^{t \mathcal{L}} u_0) \|_{\dot{B}_{p, \infty}^0}
\\
&= \sup_{j \in \mathbb{Z}} ( \| \phi_j(\sqrt{-\mathcal{L}}) t^{1/2} (-\mathcal{L})^{-1/2} \text{div} \, \sigma( e^{t \mathcal{L}} u_0) \|_{L^p} )
\\
&\leq C \sup_{j \in \mathbb{Z}} ( \| \phi_j(\sqrt{-\mathcal{L}}) (-\mathcal{L})^{-1/2} \text{div} \|_{L^p \to L^p} 
 \| t^{1/2} \nabla e^{t \mathcal{L}} u_0 \|_{L^p} )
\\
&\leq C \sup_{j \in \mathbb{Z}} ( \| \nabla  (-\mathcal{L})^{-1/2} \phi_j(\sqrt{-\mathcal{L}}) \|_{L^{p'} \to L^{p'}} 
 \| t^{1/2} \nabla e^{t \mathcal{L}} u_0 \|_{L^p} )
\end{aligned}
\end{equation*}

For any $1 \leq p \leq \infty, j \in \mathbb{Z}$, 
Theorem~\ref{thm:1} in this paper and Lemma~2.1 in~\cite{paper:Iw-2018} imply estimates that are uniform in $j$.
\begin{equation*}
\begin{aligned}
&\quad \| \nabla (-\mathcal{L})^{-1/2} \phi_j(\sqrt{-\mathcal{L}}) \|_{L^{p'} \to L^{p'}}
\\
&\leq \| \nabla e^{2^{-2j} \mathcal{L}} e^{-2^{-2j} \mathcal{L}} (-\mathcal{L})^{-1/2} \phi_j(\sqrt{-\mathcal{L}}) \|_{L^{p'} \to L^{p'}}
\\
&\leq \| \nabla e^{2^{-2j} \mathcal{L}} \|_{L^{p'} \to L^{p'}} 
  \| e^{-2^{-2j} \mathcal{L}} (-\mathcal{L})^{-1/2} \phi_j(\sqrt{-\mathcal{L}}) \|_{L^{p'} \to L^{p'}} 
\\
&\leq C 2^j \| e^\lambda (2^{2j} \lambda)^{-1/2} \phi_0(\sqrt{\lambda}) \|_{H^\beta(\mathbb{R})}
\\
&\leq C 2^j \cdot 2^{-j} 
\\
&\leq C.
\end{aligned}
\end{equation*}
Here, $\beta$ is a real number satisfying $\beta > (d+1) / 2$.
From the above and~\eqref{eq:20260322-2}, the following holds.
\begin{equation*}
\begin{aligned}
\| u_0 \|_{\dot{B}_{p, q}^s}
&\leq C \left\{ \int_0^\infty (t^{-s/2} \| t^{1/2} \nabla e^{t \mathcal{L}} u_0 \|_{L^p} )^q 
\frac{1}{t} \, \mathrm{d}t \right\}^{1/q}.
\end{aligned}
\end{equation*}

We briefly describe the case of a general $\alpha$.
If $\alpha$ is even ($\alpha = 2k$), then inequality~\eqref{eq:20260322-2} can be transformed as follows.
\begin{equation*}
\begin{aligned}
\| u_0 \|_{\dot{B}_{p, q}^s}
&\leq C \left\{ \int_0^\infty (t^{-s/2} \| (-t\mathcal{L})^{\alpha/2} e^{t \mathcal{L}} u_0 \|_{\dot{B}_{p, \infty}^0})^q 
\frac{1}{t} \, \mathrm{d}t \right\}^{1/q}
\\
&\leq C \left\{ \int_0^\infty (t^{-s/2} \| t^{\alpha/2} \nabla^{\alpha} e^{t \mathcal{L}} u_0 \|_{\dot{B}_{p, \infty}^0})^q 
\frac{1}{t} \, \mathrm{d}t \right\}^{1/q}.
\end{aligned}
\end{equation*}
On the other hand, if $\alpha$ is odd ($\alpha = 2k + 1$), it suffices to rewrite inequality~\eqref{eq:20260322-2} as follows
and then apply the preceding estimate.
\begin{equation*}
\begin{aligned}
\| u_0 \|_{\dot{B}_{p, q}^s}
&\leq C \left\{ \int_0^\infty (t^{-s/2} \| (-t\mathcal{L})^{\alpha/2} e^{t \mathcal{L}} u_0 \|_{\dot{B}_{p, \infty}^0})^q 
\frac{1}{t} \, \mathrm{d}t \right\}^{1/q}
\\
&\leq C \left\{ \int_0^\infty (t^{-s/2} \| t^{\alpha/2} (-\mathcal{L})^{-1/2} \text{div} \, \sigma( \mathcal{L}^k e^{t \mathcal{L}} u_0) \|_{\dot{B}_{p, \infty}^0})^q 
\frac{1}{t} \, \mathrm{d}t \right\}^{1/q}
\\
&\leq C \left\{ \int_0^\infty (t^{-s/2} \| t^{\alpha/2} \nabla^{\alpha} e^{t \mathcal{L}} u_0 \|_{\dot{B}_{p, \infty}^0})^q 
\frac{1}{t} \, \mathrm{d}t \right\}^{1/q}.
\end{aligned}
\end{equation*}

\noindent \textit{The second inequality of~\eqref{eq:20260205-2}.}
Theorem~\ref{thm:1} in this paper implies that
\begin{equation*}
\begin{aligned}
\| t^{\alpha/2} \nabla^\alpha e^{t \mathcal{L}} u_0 \|_{L^p} 
&= t^{\alpha/2} \| \nabla^\alpha e^{t\mathcal{L} /2} e^{t\mathcal{L}/2} u_0 \|_{L^p}
\\
&\leq C \| e^{t\mathcal{L}/2} u_0 \|_{L^p}
\end{aligned}
\end{equation*}
and therefore, 
\begin{equation*}
\begin{aligned}
\quad& \left\{ \int_0^\infty (t^{-s/2} \| t^{\alpha/2} \nabla^\alpha e^{t \mathcal{L}} u_0 \|_{L^p} )^q 
\frac{1}{t} \, \mathrm{d}t \right\}^{1/q}
\\
&\leq C \left\{ \int_0^\infty (t^{-s/2} \|e^{t\mathcal{L}/2} u_0 \|_{L^p} )^q 
\frac{1}{t} \, \mathrm{d}t \right\}^{1/q}
\\
&\leq C \| u_0 \|_{\dot{B}_{p, q}^s}.
\end{aligned}
\end{equation*}

\end{document}